\documentclass{amsart}
\usepackage{amsmath}
\usepackage{amssymb}
\usepackage{amsfonts}
\usepackage[dvips]{graphicx,color}
\usepackage{amsthm}
\usepackage{amscd}
\usepackage{amsfonts, braket}
\usepackage[all]{xy}

\DeclareMathOperator{\Hom}{Hom}

\DeclareMathOperator{\Ker}{Ker}

\DeclareMathOperator{\Int}{Int}
\DeclareMathOperator{\Aut}{Aut}

\DeclareMathOperator{\Sp}{Sp}

\DeclareMathOperator{\ad}{ad}
\DeclareMathOperator{\Der}{Der}

\DeclareMathOperator{\aut}{aut}

\DeclareMathOperator{\sgn}{sgn}
\DeclareMathOperator{\Alt}{Alt}

\def\r{\mathbb{R}}
\def\c{\mathbb{C}}
\def\D{\mathfrak{D}}

\def\q{\mathbb{Q}}

\def\C{\mathcal{C}}

\def\z{\mathbb{Z}}

\def\cyc{\text{cyc}}

\def\id{\mathrm{id}}

\def\pt{\partial}
\def\ad{\text{ad}}

\def\ct{\hat{L}W}

\def\Ccom{C_{\text{com}}^{\bullet,\bullet}}
\def\Cass{C_\text{ass}^{\bullet,\bullet}}
\def\C{\Ccom(W)}
\def\tC{\hat{C}^{\bullet,\bullet}_{\text{com}}(W)}

\def\CE{C^{\bullet,\bullet}_{CE}}
\def\Sh{\mathrm{Sh}}
\def\Ush{\mathrm{Ush}}
\def\G{\mathcal{G}}

\def\com{\text{com}}
\def\cct{\hat{T}W}

\theoremstyle{definition}
  \newtheorem{thm}{Theorem}[section]
  
  \newtheorem{lem}[thm]{Lemma}

  \newtheorem{prop}[thm]{Proposition}

\theoremstyle{definition}
  \newtheorem{defi}[thm]{Definition}

  \newtheorem{rem}[thm]{Remark}
  \newtheorem{ex}[thm]{Example}

\title{Double graph complex and characteristic classes of fibrations}

\author{Takahiro Matsuyuki}
\address{Department of Mathematics, 
Tokyo Institute of Technology, 
2-12-1 Oh-okayama, Meguro-ku, Tokyo 152-8551, Japan.}
\email{matsuyuki.t.aa@m.titech.ac.jp}

\begin{document}
\maketitle
\begin{abstract}
In this paper, we construct a double chain complex generated by certain graphs and a chain map from that to the Chevalley-Eilenberg double complex of the dgl of symplectic derivations on a free dgl. It is known that the target of the map is related to characteristic classes of fibrations. We can describe some characteristic classes of fibrations whose fiber is a 1-punctured even-dimensional manifold by linear combinations of graphs though the cohomology of the dgl of derivations. 

\end{abstract}

\section{Introduction}

The Chevalley-Eilenberg complex of the limit of the Lie algebra of symplectic derivations on (graded) free Lie algebras is isomorphic to the graph complex defined by the cyclic Lie operad (details in \cite{K, K2, Con, Ham}). In this paper, we introduce an extension of (the dual of) the construction to a Lie algebra of symplectic derivations on free dgls. Let $(W,\omega)$ be a finite-dimensional graded vector space with symmetric inner product of even degree $N$ and $\delta$ a differential of degree $-1$ on the completed free Lie algebra $\hat{L}W$ satisfying the symplectic condition $\delta\omega=0$. An important example is the case that $(\hat{L}W,\delta)$ is a Chen's dgl model of an even dimensional manifold and $\omega$ is its intersection form. We construct a $W$-labeled graph complex $\C$ and a chain map
\[\C\to \CE(\Der_\omega(\hat{L}W))\]
to the Chevalley-Eilenberg (double) complex $\CE(\Der_\omega(\hat{L}W))$ of the differential graded Lie algebra $(\Der_\omega(\hat{L}W),\ad(\delta))$ of symplectic derivations on $\hat{L}W$. Furthermore, from the non-labeled part $\Ccom(N,Z)$ of the graph complex, which depends on only the integer $N$ and the set $Z$ of degrees of $W$, we can obtain a chain map
\[\Ccom(N,Z)\subset \C^{\Sp(W,\delta)}\to \CE(\Der_\omega(\hat{L}W))^{\Sp(W,\delta)},\]
where $\Sp(W,\delta)$ is the group of graded linear isomorphisms of $W$ preserving $\omega$ and $\delta$. In the case of $N=0$ and $Z=\{0\}$, the map corresponds to the Kontsevich's one \cite{K,K2}. 

The construction above gives characteristic classes of fibrations. It is known that characteristic classes of simply-connected fibrations are related to Lie algebras of derivations \cite{SS,Tanre}. In non-simply connected cases, we got relations between characteristic classes and Lie algebras of derivations as in \cite{MT,KMT}. In this paper, we consider the case that the boundary of a fiber is a sphere. For a simply-connected compact manifold $X$ with $\pt X=S^{n-1}$, let $\aut_\pt(X)$ be the monoid of self-homotopy equivalences of $X$ fixing the boundary pointwisely and $\aut_{\pt,0}(X)$ its connected component containing $\id_X$. According to \cite{Ber}, the isomorphism
\[H^\bullet(B\aut_{\pt,0}(X);\q)\simeq H^\bullet_{CE}(\Der^+_\omega(L_X))\]
is obtained. Here $L_X$ is a cofibrant dgl model of $X$. The underlying Lie algebra of $L_X$ is generated by the linear dual $W$ of the suspension of the reduced cohomology of $X$. So the graph complex above gives the invariant part of the cohomology $H^\bullet_{CE}(\Der^+_\omega(L_X))$ with respect to the action of the group $\Sp(W,\delta)$ of automorphisms of $W$ with intersection form preserving the differential $\delta$ of $L_X$. Using the Serre spectral sequence for the fibration \[B\aut_{\pt,0}(X)\to B\aut_\pt(X)\to B\pi_0(\aut_\pt(X)),\] the image of the natural map $H^\bullet(B\aut_\pt(X);\q)\to H^\bullet(B\aut_{\pt,0}(X);\q)$ is included in the invariant part. We give a chain map 
\[\Ccom(N,Z)_+\to \CE(\Der^+_\omega(L_X))^{\Sp(W,\delta)}.\]
using a positive truncated version $\C_+$ of $\C$. Considering $W$-labeled graphs, we can also obtain a $W$-labeled version $\C_+$ and a chain map 
\[\C_+\to \CE(\Der^+_\omega(L_X)).\]

\noindent
{\bf Acknowledgment.} I would like to thank Y. Terashima and H. Kajiura for many helpful comments. This work was supported by Grant-in-Aid for JSPS Research Fellow (No.17J01757).

\section{Preliminary}

In this paper, all vector spaces are over a field $K$ whose characteristic is zero. A field $K$ is regarded as a $\z$-graded vector space whose all elements have degree $0$.

For a finite set $U$, the number of elements in $U$ is denoted by $\# U$.

All tensor products of linear maps between $\z$-graded vector spaces contain their signs: for homogeneous linear maps $f:A\to V$, $g:B\to W$ between $\z$-graded vector spaces, we set
\[(f\otimes g)(a,b):=(-1)^{ga}f(a)\otimes g(b)\]
for $a\in A$ and $b\in B$. (We often denote by $|a|$ the degree of an element $a$. But we omit the symbol $|\cdot|$ of the degree when it appears in a power of $-1$. For example, $(-1)^{ga}$ means $(-1)^{|g||a|}$.)

Let $V$ be a $\z$-graded vector space. We denote $V^i$ the subspace of elements of $V$ of \textbf{cohomological degree} $i$ and $V_i=V^{-i}$ the subspace of elements of \textbf{homological degree} $i$. Remark that the \textbf{linear dual} $V^*=\Hom(V,\r)$ of $V$ is graded by $(V^*)^i=\Hom(V_i,\r)$.

The \textbf{$p$-fold suspension} $V[p]$ of $V$ for an integer $p$ is defined by 
\[V[p]^i:=V^{i+p}\]
and elements of $V[p]^i$ are presented by $x\sigma$ for $x\in V^{i+p}$ using the symbol $\sigma$ of cohomological degree $-p$. The $p$-suspension map $V\to V[p]$ is also denoted by $\sigma$. In this paper, the $N$-suspension $\sigma$ for an even number $N$ often appears. It is used for adjusting degrees of elements though we can ignore it when calculating signs. 

Let $V$ be a $\z$-graded vector space and $\alpha:V\otimes V\to K$ be a non-degenerate bilinear map of (cohomological) degree $n$. Out of the two conditions
\begin{enumerate}
\item $\alpha(x,y)=(-1)^{xy}\alpha(y,x)$ for homogeneous elements $x,y\in V$, and
\item $\alpha(x,y)=-(-1)^{xy}\alpha(y,x)$ for homogeneous elements $x,y\in V$,
\end{enumerate}
the pair $(V,\alpha)$ is called \textbf{symmetric vector space} with degree $n$ if satisfying (i), and \textbf{symplectic vector space} with degree $n$ if satisfying (ii).

\subsection{Algebras and  signs}

Let $V$ be a finite-dimensional $\z$-graded vector space. 
\begin{defi}
We define the following quotient algebras of the tensor algebra $TV$ generated by $V$.
\begin{itemize}
\item The \textbf{symmetric algebra} $SV$ generated by $V$ is the $\z$-graded commutative algebra which is the quotient algebra obtained from the $\z$-graded tensor algebra $TV$ by introducing the relation
\[xy=(-1)^{xy}yx\]
for $x,y\in V$. The image of $V^{\otimes k}$ for an integer $k$ in $SV$ is denoted by $S^kV$. 
\item The \textbf{exterior algebra} $\Lambda V$ generated by $V$ is the $\z$-graded anti-commutative algebra which is the quotient algebra obtained from the $\z$-graded tensor algebra $TV$ by introducing the relation
\[xy=-(-1)^{xy}yx\]
for $x,y\in V$. The image of $V^{\otimes k}$ for an integer $k$ in $\Lambda V$ is denoted by $\Lambda^kV$. 
\end{itemize}
\end{defi}

\begin{defi}\label{Kos}
For distinct elements $v_1,\dots,v_k\in V$ and a permutation $\pi\in\mathfrak{S}_k$, the sign $\epsilon$ defined by the equation on $S^kV$
\[v_1\cdots v_k=\epsilon\cdot v_{\pi(1)}\cdots v_{\pi(k)}\]
is called the \textbf{Koszul sign} of $(v_1,\dots,v_k)\mapsto (v_{\pi(1)},\dots,v_{\pi(k)})$. Similarly the sign $\bar{\epsilon}$ defined by the same equation in $\Lambda^kV$ is called the \textbf{anti-Koszul sign}. Note that the equation $\bar{\epsilon}=\sgn \pi \cdot \epsilon$.
\end{defi}

\subsection{Derivations}
Let $W$ be a finite-dimensional $\z$-graded vector space. 
\subsubsection{Completed tensor algebras}
We denote the \textbf{completed tensor algebra} by
\[\hat{T}W:=\prod_{r=0}^\infty W^{\otimes r}.\]
Its product $\mu$ and coproduct $\Delta$ are defined by
\[\mu(x_1\otimes \cdots\otimes x_s,x_{s+1}\otimes \cdots \otimes x_r)=x_1\otimes\cdots \otimes x_r,\]
\[ \Delta(x_1\otimes\cdots \otimes x_r)=\sum_{s=0}^r\sum_{\tau\in \Ush(s,r-s)}\epsilon \cdot(x_{\tau(1)}\otimes \cdots\otimes x_{\tau(s)})\otimes (x_{\tau(s+1)}\otimes \cdots\otimes x_{\tau(r)})\]
for homogeneous elements $x_1,\dots,x_r\in W$, where $\Ush(s,r-s)$ is the set of $(r,s-r)$-unshuffles and $\epsilon$ is the Koszul sign of the permutation $(x_1,\dots,x_r)\mapsto (x_{\tau(1)} ,\dots, x_{\tau(r)})$ (Definition \ref{Kos}). The primitive part of $\hat{T}W$ is the completed free Lie algebra $\hat{L}W$. These algebras have the gradings defined by the grading of $W$.

\subsubsection{Derivations on a completed tensor algebra}\label{der-tensor}
Let $\Der(\ct)$ be the Lie algebra of (continuous) derivations on the completed algebra $\hat{T}W$. Given a symplectic inner product $\omega$ of degree $N$ on $W$, we define the Lie algebra of \textbf{symplectic derivations} on $\hat{T}W$
\[\Der_\omega(\hat{T}W):=\{D\in \Der(\hat{T}W);\ D(\omega)=0\}.\]
Here $\omega$ is identified with the element of $\hat{L}W$ described by
\[\sum_{i<j}\omega_{ij}[x^i,x^j],\]
where $\{x^i\}$ is a basis of $W$ and the matrix $(\omega_{ij})_{i,j}$ is the inverse matrix of $(\omega(x^i,x^j))_{i,j}$.

Since derivations on $\hat{T}W$ are determined by the values on the generating space $W$, we get the isomorphism as graded vector space
\[\Phi_\omega:\Der(\hat{T}W)\simeq \Hom(W,\hat{T}W)\simeq \hat{T}W\otimes W[-N]=\prod_{r=1}W^{\otimes r}[-N],\]
where the second isomorphism is induced by the isomorphism $\Hom(W,\r)\simeq W[-N]$ derived from non-degeneracy of $\omega$. Furthermore, we also have the identification by $\Phi_\omega$
\begin{align*}\Der^r(\cct)&:=\{D\in \Der(\cct);D(W)\subset W^{\otimes (r+1)}\}\\
&\simeq \Hom(W,W^{\otimes (r+1)})\simeq W^{\otimes (r+2)}[-N].\end{align*}
Fixing a homogeneous basis $x^1,\dots,x^m$ of $W$, the derivations $x^{i_1}\cdots x^{i_k}\pt/\pt x^i$ ($1\leq i_1,\dots,i_k,i\leq m$), which are these elements corresponding to the linear map $x^i\mapsto x^{i_1}\cdots x^{i_k}$, consist a basis of $\Der^{k+1}(\cct)$. On the basis, $\Phi_\omega$ is described by
\[\Phi_\omega\left(x^{i_1}\cdots x^{i_k}\frac{\pt}{\pt x^i}\right)=\sum_{j}\omega_{ij} x^{i_1}\cdots x^{i_k}x^j\sigma^{-1},\]
where $\sigma^{-1}$ is a symbol of the $(-N)$-suspension which has homological degree $-N$. 

By the identification $\Phi_\omega$, the space of symplectic derivations is described by 
\[\Der_\omega(\hat{T}W)\overset{\Phi_\omega}{\simeq}  \prod_{r=1}^\infty(W^{\otimes r})^{\z/r\z}[-N]=\prod_{r=1}^\infty W_\cyc^{(r)}[-N].\]
Here $W_\cyc^{(r)}:=(W^{\otimes r})^{\z/r\z}$ is the space of invariant tensors by cyclic permutations of tensor factors, which is also defined in Definition \ref{cyclic}.

Therefore the Lie algebra $\Der_\omega(\ct)$ of symplectic derivations on $\ct$ is described by
\[\Der_\omega(\ct):=\Der(\ct)\cap \Der_\omega(\hat{T}W)\overset{\Phi_\omega}{\simeq}  \prod_{r=2}^\infty W(r)[-N],\]
\[\Der_\omega^{r+2}(\ct):=\Der^{r+2}(\ct)\cap \Der_\omega(\hat{T}W)\overset{\Phi_\omega}{\simeq}  W(r)[-N],\]
where $W(r):=(LW\otimes W)\cap W_\cyc^{(r)}$.

Through the isomorphism $\Phi_\omega$, the Lie algebra structure of $\Der(\cct)$ is described as follows:
\begin{lem}\label{derstr}
Let $[\ ,\ ]$ be the Lie bracket of $\Der(\cct)$. Then the linear map $[\ ,\ ]_\omega:=\sigma\Phi_\omega\circ [\ ,\ ]\circ (\Phi_\omega^{-1}\sigma^{-1})^{\otimes 2}$ is equal to 
\[\sum_{d_1+d_2=N}(\id\otimes \omega_{(d_1,d_2)})\left(\sum_{1\leq t<r_2}\pi_{1;t}^{r_1,r_2}+\sum_{1\leq s<r_1}\pi_{2;s}^{r_1,r_2}\right):W^{\otimes r_1}\otimes W^{\otimes r_2}\to W^{\otimes r_1+r_2-2}\]
where $\omega_{(d_1,d_2)}:W\otimes W\to\r$ for integers $d_1,d_2$ is the composition of the projection $W\otimes W\to W_{d_1}\otimes W_{d_2}$ and the restriction of $\omega$ to $W_{d_1}\otimes W_{d_2}$, and $\pi_{1;j}^{r_1,r_2},\pi_{2;i}^{r_1,r_2}:W^{\otimes r_1}\otimes W^{\otimes r_2}\to W^{\otimes r_1+r_2}$ for $1\leq i\leq r_1$, $1\leq j\leq r_2$ is defined by 
\[\pi_{1;j}^{r_1,r_2}(a_{1}^{(1)}\cdots a_{r_1}^{(1)}\otimes a_{1}^{(2)}\cdots a_{r_2}^{(2)})=\epsilon\cdot 
a_{1}^{(2)}\cdots a_{j-1}^{(2)}a_{1}^{(1)}\cdots a_{r_1-1}^{(1)}a_{j+1}^{(2)}
\cdots a_{r_2}^{(2)}a_{r_1}^{(1)}a_j^{(2)}\]
\[\pi_{2;i}^{r_1,r_2}(a_{1}^{(1)}\cdots a_{r_1}^{(1)}\otimes a_{1}^{(2)}\cdots a_{r_2}^{(2)})=\epsilon\cdot 
a_{1}^{(1)}\cdots a_{i-1}^{(1)}a_{1}^{(2)}\cdots a_{r_2-1}^{(2)}a_{j+1}^{(2)}
\cdots a_{r_2-1}^{(2)}a_{i}^{(1)}a_{r_2}^{(2)}\]
for homogeneous elements $a_1^{(1)},\dots,a_{r_1}^{(1)},a_1^{(2)},\dots,a_{r_2}^{(2)}$. Here $\epsilon$ is the Koszul sign of the corresponding permutations.
\end{lem}
\begin{proof}
Let $x^1,\dots,x^m$ be a homogeneous basis of $W$. The Lie bracket for the basis is described by
\[\left[x^{i_1}\cdots x^{i_k}\frac{\pt}{\pt x^i},x^{j_1}\cdots x^{j_l}\frac{\pt}{\pt x^j}\right]\]\[=\sum_t\epsilon\delta_i^{j_t} x^{j_1}\cdots x^{j_{t-1}}x^{i_1}\cdots x^{i_k}x^{j_{t+1}}\cdots x^{j_l}\frac{\pt}{\pt x^j}-\sum_s\epsilon'\delta_j^{i_s}x^{i_1}\cdots x^{i_{s-1}}x^{j_1}\cdots x^{j_l}x^{i_{s+1}}\cdots x^{i_k}\frac{\pt}{\pt x^i}\]
where $\epsilon=(-1)^{(x^{i_1}+\cdots +x^{i_k}-x^i)(x^{j_{1}}+\cdots +x^{j_{t-1}})}$, $\epsilon'=(-1)^{(x^{j_1}+\cdots +x^{j_l}-x^j)(x^{i_{s+1}}+\cdots +x^{i_k}-x^i)}$, and $\delta^i_j$ is the Kronecker's delta.

Then, for $A=x^{i_1}\cdots x^{i_{r_1}}$ and $B=x^{j_1}\cdots x^{j_{r_2}}$, we obtain 
\begin{align*}
[ A, B]_\omega=&\sum_t\epsilon x^{j_1}\cdots x^{i_1}\cdots x^{i_{r_1-1}}\cdots x^{j_{r_2}}\omega^{i_{r_1}j_t}\\
&+\sum_s\epsilon' x^{i_1}\cdots x^{j_1}\cdots x^{j_{r_2-1}}\cdots x^{i_{r_1}}\omega^{i_sj_{r_2}}\\
=&\sum_{d_1+d_2=N}(\id\otimes\omega_{(d_1,d_2)})\left(\sum_{1\leq t<r_2}\pi_{1;t}^{r_1,r_2}+\sum_{1\leq s<r_1}\pi_{2;s}^{r_1,r_2}\right)(A\otimes B),
\end{align*}
where $\epsilon$ and $\epsilon'$ are the Koszul sign of 
\begin{align*}
(x^{i_1},\dots, x^{i_{r_1}},x^{j_1},\dots, x^{j_{r_2}})\mapsto (x^{j_1},\dots, x^{i_1},\dots, x^{i_{r_1-1}},\dots, x^{j_{r_2}},x^{i_{r_1}},x^{j_t}),\\ 
(x^{i_1},\dots ,x^{i_{r_1}},x^{j_1},\dots, x^{j_{r_2}})\mapsto (x^{i_1},\dots, x^{j_1},\dots, x^{j_{r_2-1}},\dots, x^{i_{r_1}},x^{i_s},x^{j_{r_2}})\end{align*}
respectively. In the calculus above, note that we use the assumption that $N$ is even.
\end{proof}

The lemma above is needed to prove Theorem \ref{graphmap}.

\subsubsection{Derivations on a dgl}
Let $\delta$ be an element in $\Der_\omega(\ct)$ of homological degree $-1$ such that $\delta^2=0$. Then $\ad(\delta)$ is a differential operator on $\Der_\omega(\ct)$. 

In the case that $W_0$ is positively graded, i.e., $W_i=0$ for $i\leq 0$, we can regard $\delta\in \Der_\omega(LW)$ since $\delta$ is described by only finite sums. Then we often consider the positive truncation $(\Der_\omega^+(LW),\ad(\delta))$ of the chain complex $(\Der_\omega(LW),\ad(\delta))$ defined by
\[\Der^+_\omega(LW)_i:=\begin{cases}\Der_\omega(LW)_i&(i>2)\\ \Ker(\ad(\delta))_1&(i=1)\\0&(\text{otherwise}).\end{cases}\]

\begin{defi}[Chevalley-Eilenberg complex]
Let $(L,\delta)$ be a dgl. We define the Chevalley-Eilenberg complex as follows:
\[C^{p,q}_{CE}(L):=(\Lambda^pL^*)^q,\]
where $\Lambda^\bullet L^*$ is the exterior algebra generated by the graded vector space $L^*$. The first differential $d_{CE}$ is defined by the formula for $c\in C^{p,q}_{CE}(L)$ and $D_1,\dots,D_{p+1}\in L$, 
\[(d_{CE}c)(D_1,\dots,D_{p+1})=\sum_{i<j}\bar{\epsilon}\cdot c([D_i,D_j],D_1,\dots,\hat{D}_i,\dots,\hat{D}_j,\dots,D_{p+1}),\]
where $\bar{\epsilon}=(-1)^{D_i(D_1+\cdots+D_{i-1})+D_j(D_1+\cdots +D_{j-1})+D_iD_j+i+j-1}$, and the second differential $L_\delta$ derived from $\delta$ is defined by
\[L_\delta=i_\delta d_{CE}-d_{CE}i_\delta,\]
using the \textbf{interior product} defined by
\[(i_\delta c)(D_1,\dots,D_{p})=c(\delta,D_1,\dots,D_{p}),\]
for $c\in C^{p+1,q}_{CE}(L)$ and $D_1,\dots,D_{p}\in L$. Then the triple $(\CE(L),d_{CE},L_\delta)$ is a double complex.
\end{defi}

In this paper, we consider the Chevalley-Eilenberg complexes of dgls $(\Der_\omega(\ct),\ad(\delta))$ and $(\Der^+_\omega(LW),\ad(\delta))$, and the invariant space $\CE(\Der^+_\omega(LW))^{\Sp(W,\delta)}$, where $\Sp(W,\delta)$ is the group of symplectic linear isomorphisms $W\to W$ preserving $\delta$. 

\subsection{dgl model with symplectic form of manifolds}

In this subsection, we review a Chen's dgl model of a manifold. Let $X$ be a smooth manifold. Put $A=A^\bullet(X)$ and $H=H_{DR}^\bullet(X)$. Fix a homotopy transfer diagram 
\[\xymatrix{\ar@(ld,lu)[] A\ar@<0.5ex>[r]&H,\ar@<0.5ex>[l]}\]
e.g. in the case that $X$ is a closed manifold, it is obtained by using the Hodge decomposition of the de Rham complex $A$. Since $A$ is a commutative dga with symmetric form (intersection form), $H$ has the structure of minimal cyclic $C_\infty$-algebra by the diagram (details in \cite{KoSo, Me, Kad, MSS, HLNC} for instance). 

Let $I$ be the intersection form on $H$, $m$ the cyclic $C_\infty$-algebra structure on $H$ obtained by the homotopy transfer diagram and $s:H\to H[1]$ be the suspension map. We denote $V=H[1]^*$. Defining the suspension of $m_i$ by $\bar{m}_i:=s\circ m_i\circ (s^{-1})^{\otimes i}$ for all $i\geq 1$ and of $I$ by $\omega:=I\circ (s^{-1})^{\otimes2}$, then the duals of these define the symplectic inner product $\omega$ on $H[1]^*$ of degree $N=n-2$ and the linear map $\bar{\delta}_i:V\to V^{\otimes n}$ of homological degree $-1$ . Thus extending the unique derivation $\bar{\delta}_i:\hat{L}V\to \hat{L}V$ by the Leibniz rule, then we have the derivation of homological degree $-1$
\[\bar{\delta}:=\sum_{i=1}^\infty\bar{\delta}_i\in \Der_\omega(\hat{L}V).\]
Furthermore we can prove that $\bar{\delta}$ is a differential since $m$ satisfies the $A_\infty$-relations and \textbf{quadratic}, i.e. $\bar{\delta}(V)\subset \prod_{i\geq2}V^{\otimes i}$, since $(H,I,m)$ is minimal. 

The Chen's dgl model is a reduced version of the construction. Suppose $X$ is connected and put
\[W:=H[1]^*_{\geq 0}=H_+(X;\r)[-1].\]
Then we have the restriction $\delta:\hat{L}W\to \hat{L}W$ of $\bar{\delta}$ and $\omega:W^{\otimes2}\to\r$. If $X$ is simply-connected, we can restrict the differential $\delta$ on the free Lie algebra $LW\subset \hat{L}W$ since $\delta(w)$ for $w\in W$ has only finitely many nontrivial terms.

\begin{thm}[Chen\cite{Chen}]
For a simply-connected closed manifold $X$ with base point $*$, the dgl $(LW,\delta)$ is a Quillen model of $X$, i.e., there is a Lie algebra isomorphism
\[H_\bullet(LW,\delta)\simeq \pi_\bullet(\Omega X)\otimes\q.\]

\end{thm}

\section{Graph complex}

\subsection{Orientation and ordering of graded sets}

\def\Ord{\mathrm{Ord}}
\def\Or{\mathrm{Or}}
\def\Cyc{\mathrm{Cyc}}
The set of \textbf{orderings} on a set $U$ is defined by
\[\Ord(U):=\{(u_1,\dots,u_k)\in U^{\times k};U=\{u_1,\dots,u_k\}\},\]
where $k:=\#U$. 

\begin{defi}
Let $U$ be a $\z$-graded set, i.e. a finite set $U$ given a map $|\cdot|:U\to\z$.
\begin{itemize}
\item The graded vector space generated by $U$ is denoted by $\r U$.
\item The symmetric algebra generated by $U$ is denoted by $SU:=S(\r U)$.
\item The exterior algebra generated by $U$ is denoted by $\Lambda U:=\Lambda(\r U)$.
\end{itemize}
For an element $(u_1,\dots,u_k)\in \Ord(U)$, we denote the image of $u_1\otimes \cdots \otimes u_k$ in $\Lambda U$ by $[u_1,\dots,u_k]$. The 1-dimensional vector space generated by this element is written by\[ O(U):=\braket{[u_1,\dots,u_k]}\subset \Lambda U.\] 
\end{defi}

\begin{defi}\label{cyclic}
Let $V$ be a $\z$-graded vector space. We define the subspace $V_\cyc^{(k)}$ of \textbf{cyclic tensors} in $V^{\otimes k}$ by the image of the map $[-,\dots,-]_\cyc:V^{\otimes k}\to V^{\otimes k}$ obtained by
\[x_1\otimes \cdots\otimes x_k\mapsto \sum_{\tau\in \z/k\z}\epsilon\cdot x_{\tau(1)}\otimes\cdots \otimes x_{\tau(k)},\]
where $\z/k\z$ is identified with the group of cyclic permutations and $\epsilon$ is the Koszul sign of $(x_1,\dots,x_k)\mapsto (x_{\tau(1)},\dots,x_{\tau(k)})$. 
For a $\z$-graded set $U$, we denote
\[\Cyc(U) :=\braket{[u_1,\dots, u_k]_\cyc;(u_1,\dots,u_k)\in \Ord(U)}\subset (\r U)_\cyc^{(k)}.\]
\end{defi}

\subsection{Definition of graph complex}\label{defgraph}
Let $W$ be a finite-dimensional symplectic vector space with form $\omega$ of degree $N$ and suppose that $N$ is even and $Z:=\{a\in \z;W_a\neq 0\}\subset \{0,\dots,N\}$. Our labeled graph complex depends on $(W,\omega)$.

\subsubsection{Definition of graphs}
\begin{defi}
An \textbf{$N$-graded graph} $\Gamma$ consists of the following information:
\begin{itemize}
\item The set $H(\Gamma)$ of \textbf{half-edges}.
\item The set $V(\Gamma)$ of \textbf{vertices}. It is a partition of the set $H(\Gamma)$, i.e.
\[H(\Gamma)=\coprod_{v\in V(\Gamma)}v,\quad v\neq \emptyset\ (v\in V(\Gamma)).\]
The number $\# v$ of elements of any $v\in V(\Gamma)$ is called the \textbf{valency} of $v$. A vertex with valency $>1$ is called an \textbf{internal vertex} and one with valency $1$ is called an \textbf{external vertex}. The set of internal (resp. external) vertices is denoted by $V_i(\Gamma)$ (resp. $V_e(\Gamma)$).
\item The set $E(\Gamma)$ of \textbf{edges}. It is a partition of the set $H(\Gamma)$ such that the number of elements of any $e\in E(\Gamma)$ is two, i.e.
\[H(\Gamma)=\coprod_{e\in E(\Gamma)}e,\quad \# e=2\ (e\in E(\Gamma)).\]
\item The cohomological \textbf{degree of half-edges}. It is a map $|\cdot|:H(\Gamma)\to Z$ such that $|h_1|+|h_2|=N$ for an edge $e=\{h_1,h_2\}\in E(\Gamma)$. Then the cohomological degrees of vertices and edges are defined by
\[|v|:=|h_1|+\cdots+|h_r|-N,\quad |e|:=N\]
for $v=\{h_1,\dots,h_r\}\in V(\Gamma)$ and $e\in E(\Gamma)$.
\item The division of the set $V_i(\Gamma)$ of internal vertices to two disjoint sets \[V_i(\Gamma)=V_n(\Gamma)\amalg V_s(\Gamma)\]
such that all elements in $V_s(\Gamma)$ have cohomological degree $-1$ and the valency $\geq 3$. An element of $V_n(\Gamma)$ is called \textbf{normal vertex}, and one of $V_s(\Gamma)$ is called \textbf{special vertex}.
\end{itemize}
The set of isomorphism classes of such graphs is denoted by $\G(N)$. Here an isomorphism between $N$-graded graphs is a bijection between the sets of half-edges preserving all information of $N$-graded graphs.
\end{defi}

\begin{ex}
In the case of $N=4$ and $Z=\{0,1,2,3,4\}$, we can give examples of $4$-graded graphs in Figure \ref{ex}.  In these figures, 
\begin{itemize}
\item a black vertex $\bullet$ means a normal vertex, a white vertex $\circ$ a special vertex and a square vertex $\blacksquare$ a univalent vertex, and
\item a number drawn beside a half-edge is its degrees.
\end{itemize}
\begin{figure}[h]
\centering
\includegraphics[width=8cm]{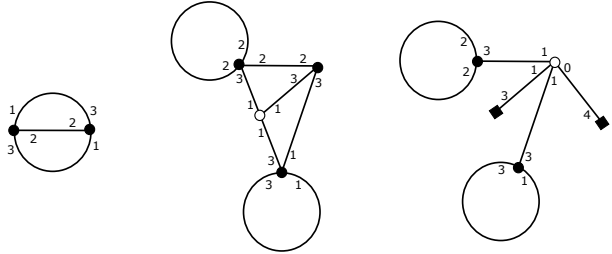}
\caption{Examples of $4$-graded graphs}\label{ex}
\end{figure}%
\end{ex}

\subsubsection{Decoration on vertices}
We shall give the relation equivalent to the dual of vertices defined by the cyclic Lie operad as in \cite{Con, Ham, Markl}.
\begin{defi}
Let $\Gamma$ be an $N$-graded graph. 
\begin{itemize}
\item We introduce to $\Cyc(v)[N]$ for $v\in V_i(\Gamma)$ the \textbf{commutativity relation} 
\[S_{v,h_r;s}(o):=\sum_{\tau\in \Sh(s,r-s-1)}o^{\tau^{(v,h_r)}}=0,\]\[ o^{\tau^{(v,h_r)}}:=\epsilon[h_{\tau(1)},\dots,h_{\tau(r-1)},h_r]_\cyc\sigma,\]
for $r-1>s>0$ and $o=[h_1,\dots,h_r]_\cyc\sigma\in\Cyc(v)[N]$, where $\Sh(p,q)$ is the set of $(p,q)$-shuffles, $\sigma$ is the symbol of the $N$-fold suspension, and $\epsilon$ is the Koszul sign. Then we denote the obtained space by $C(v)=\Cyc(v)[N]/(\text{com. rel.})$. (In the case of $r=3$, it is the AS-relation for Jacobi diagrams.) 
\end{itemize}

\end{defi}
\begin{figure}[h]
\centering
\includegraphics[width=9cm]{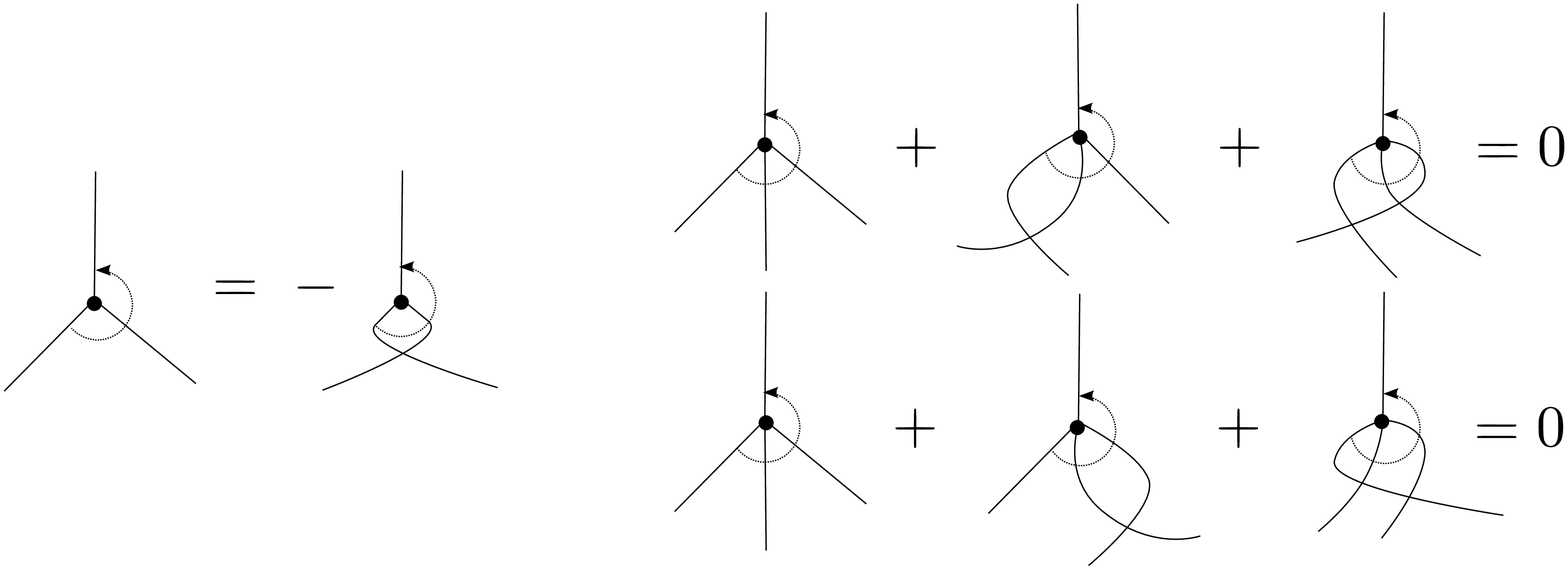}
\caption{Commutativity ($r=3,4$). (Koszul signs are omitted in figures.)}
\end{figure}%

\subsubsection{Decoration on $N$-graded graphs}\label{orientation} Set
\[\tilde{O}_\com(W,\Gamma):=\bigodot_{e\in E(\Gamma)}O(e)\otimes\bigodot_{u\in V_e(\Gamma)}W[-N]_{|u|}\otimes\bigwedge_{v^s\in V_s(\Gamma)}C(v^s) \otimes \bigwedge_{v\in V_n(\Gamma)}C(v), \]
where 
\[\bigodot_{u\in U}V(u):=\left\{v_{u_1} \cdots v_{u_k}\in S^k \left(\bigoplus_{u\in U}V(u)\right);v_{u_i}\in V({u_i}),\ (u_1,\dots,u_k)\in \Ord(U)\right\},\]
\[\bigwedge_{u\in U}V(u):=\left\{v_{u_1} \cdots v_{u_k}\in \Lambda^k \left(\bigoplus_{u\in U}V(u)\right);v_{u_i}\in V({u_i}),\ (u_1,\dots,u_k)\in \Ord(U)\right\}\]
for a family $(V(u))_{u\in U}$ of $\z$-graded vector spaces indexed by a finite set $U$. This tensor product consists of four factors: the first factor means directions of edges of $\Gamma$, the second factor $W$-labels of  external vertices of $\Gamma$, the third factor (equivalence classes of) cyclic orderings on special vertices of $\Gamma$, and the fourth factor the same on normal vertices of $\Gamma$. Note that $W[-N]_{|u|}=W_{|h|}[-N]$ for an external vertex $u=\{h\}$.

We need to identify elements of $\tilde{O}_\com(W,\Gamma)$ by the symmetry of $\Gamma$. An automorphism $\alpha$ of an $N$-graded graph $\Gamma\in \G(N)$ induces the linear isomorphism $C(v)\to C(\alpha(v))$ for $v\in V_i(\Gamma)$ described by
\[[h_1,\dots,h_k]_\cyc\mapsto [\alpha(h_1),\dots,\alpha(h_k)]_\cyc,\]
and the identity map $W[-N]_{|u|}\to W[-N]_{|\alpha(u)|}=W[-N]_{|u|}$ for $u\in V_e(\Gamma)$. Therefore the automorphism group of $\Gamma$ acts on the vector space $\tilde{O}_\com(W,\Gamma)$ by the induced permutation of half-edges. Then the coinvariant vector space of $\tilde{O}_\com(W,\Gamma)$ by this action is denoted by $O_\com(W,\Gamma)$. 
We often consider an element $o$ of $O_\com(W,\Gamma)$ described by the form
\[o=[o_1,\dots,o_l;w_1,\dots,w_{k_e};c_1^s,\dots,c_{k_s}^s;c_1,\dots,c_{k_n}].\]
\[:=(o_1\cdots o_l)\otimes(w_1\cdots w_{k_e})\otimes(c_1^s\cdots c_{k_s}^s)\otimes (c_1\cdots c_{k_n})\]
where $w_i\in W[-N]_{|u_i|}$ and \[o_i=[\hat{o}_i],\quad c_i^s=[\hat{c}_i^s]_\cyc\sigma,\quad c_i=[\hat{c}_i]_\cyc\sigma,\]
for $\hat{o}_i\in \Ord(e_i)$, $\hat{c}_i\in \Ord(v_i)$ and $\hat{c}_i^s\in\Ord(v_i^s)$. Such element $o$ is called an \textbf{orientation} of $\Gamma$, a pair $(\Gamma,o)$ is an \textbf{oriented graph}, and the information 
\[\hat{o}=(\hat{o}_1,\dots,\hat{o}_l; w_1,\dots,w_{k_e};\hat{c}_1^s,\dots,\hat{c}_{k_s}^s;\hat{c}_1,\dots,\hat{c}_{k_n})\]
is called a \textbf{lift} of an orientation $o=[\hat{o}]$ on $\Gamma$. The vector space $O_\com(W,\Gamma)$ is generated by orientations. 

\begin{ex}
In the case of $N=4$ and $Z=\{0,1,2,3,4\}$, we can give examples of decorated $4$-graded graphs in Figure \ref{ex1} and \ref{ex2}.  In these figures, 
\begin{itemize}
\item an arrow on an edge means a direction, and
\item an arc drawn around a vertex is an ordering of half-edges incident to this vertex.
\end{itemize}
\begin{figure}[h]
\centering
\includegraphics[width=10cm]{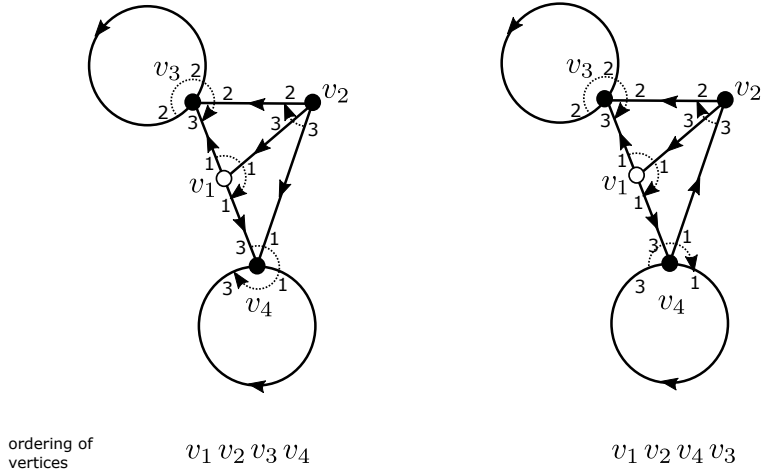}
\caption{Non-labeled examples: the left $(\Gamma,o_1)$ and the right $(\Gamma,o_2)$}\label{ex1}
\end{figure}%
\begin{figure}[h]
\centering
\includegraphics[width=6cm]{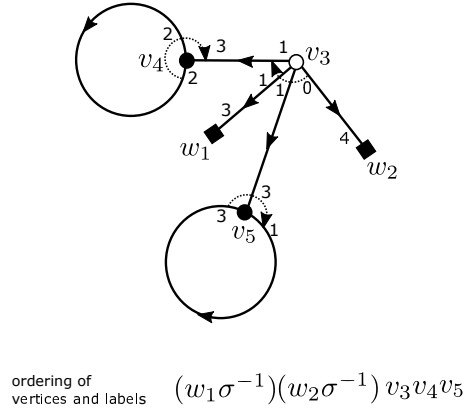}
\caption{A labeled example}\label{ex2}
\end{figure}%

In Figure \ref{ex1}, the degrees of vertices are $v_1=-1$, $v_2=4$, $v_3=5$, and $v_4=4$. In the space $O(\Gamma)$, we have\[o_1=(-1)^{5\cdot 4+1}(-1)^{3\cdot 1+1}(-1)^{3\cdot (3+1+1)}o_2=o_2,\]
where the signs $(-1)^{5\cdot 4+1},(-1)^{3\cdot 1+1},(-1)^{3\cdot (3+1+1)}$ are coming from changes of the ordering of vertices, the direction of the edge between $v_2$ and $v_4$ and the ordering of half-edges incident to $v_4$ respectively.

In Figure \ref{ex2}, elements $w_1\in W_3$ and $w_2\in W_4$ are labels of univalent vertices $v_1,v_2$ (their names $v_1,v_2$ of vertices are omitted in the figure). Note their degrees $|v_1|=|w_1\sigma^{-1}|=-1,|v_2|=|w_2\sigma^{-1}|=0$.
\end{ex}

\subsubsection{Definition of the bigraded vector space $\tC$} The cohomological bidegree $(p,q)\in \z\times\z$ of $\Gamma\in\G(N)$ is defined by
\[p=\# V_n(\Gamma),\quad q=\sum_{v\in V_n(\Gamma)}|v|=\#V_s(\Gamma)+N(\#E(\Gamma)-\#V(\Gamma))-\sum_{u\in V_e(\Gamma)}|u|,\]
and bidegree of elements in $O_\com(W,\Gamma)$ is defined by that of $\Gamma$. We define \textbf{the space of $N$-graded ribbon graphs} by
\[\tC:=\bigoplus_{\Gamma\in \G(N)}O_\com(W,\Gamma),\quad \hat{C}^{p,q}_\com (W):=\bigoplus_{\Gamma\in \G^{p,q}(W)}O_\com(W,\Gamma),\]
where $\G^{p,q}(W)$ is the subset of $\G(N)$ consisting $N$-graded graphs of degree $(p,q)$. Then $\tC$ can be regarded as bigraded vector space. 
We often denote an element in $\tC$ corresponding to $o\in O_\com(W,\Gamma)$ for $\Gamma\in\G(N)$ by $(\Gamma,o)$.

\subsubsection{Definition of the first differential $d$} We define the linear map $d_{v;h^1,h^2}^{a,b}:O_\com(W,\Gamma)\to \tC$ for an $N$-graded graph $\Gamma\in \G(N)$, a normal vertex $v\in V_n(\Gamma)$, two distinct half-edges $h^1,h^2$ incident to $v$, $a,b\in Z$ satisfying $a+b=N$. For an order $h_1,\dots,h_r$ of half-edges incident to $v$ such that $h^1=h_r$ and $h^2=h_i$, put
\[d_{v;h^1,h^2}^{a,b}(\Gamma,[-;-;-;[h_1,\dots,h_r]\sigma,-])\]\[=(\Gamma_{v;h^1,h^2}^{a,b},[-,[h',h''];-;-;[h_1,\dots,h_i,h']\sigma,[h'',h_{i+1},\dots,h_r]\sigma,-]).\]
Here $\sigma$ is the $N$-fold suspension, and the $N$-graded graph $\Gamma_{v;h^1,h^2}^{a,b}$ is defined by
\[H(\Gamma_{v;h^1,h^2}^{a,b})=H(\Gamma)\amalg\{h',h''\},\quad V(\Gamma_{v;h^1,h^2}^{a,b})=\left(V(\Gamma)\setminus\{v\}\right)\amalg\{v',v''\},\] \[V_s(\Gamma_{v;h^1,h^2}^{a,b})=V_s(\Gamma),\quad E(\Gamma_{v;h^1,h^2}^{a,b})=E(\Gamma)\amalg\{e_0\},\]
where $v'=\{h_1,\dots,h_i,h'\}$, $v''=\{h'',h_{i+1},\dots,h_r\}$, $e_0=\{h',h''\}$, $|h'|=a$ and $|h''|=b$. Note that the equation above is enough to define the operator $d^{a,b}_{v;h^1,h^2}$ and the operator is well-defined. 
\begin{figure}[h]
\centering
\includegraphics[width=8cm]{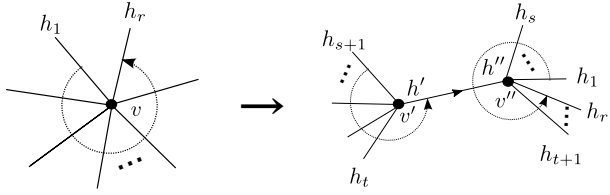}
\caption{The operator $d_{v;h_s,h_t}^{a,b}$.}
\end{figure}%

Then we obtain the linear map $d:\tC\to\tC$ by
\[d_v(\Gamma,o):=\frac12 \sum_{a+b=N}\sum_{h^1\neq h^2\in v}d_{v;h^1,h^2}^{a,b}(\Gamma,o),\quad d(\Gamma,o):=\sum_{v\in V_n(\Gamma)}d_v(\Gamma,o).\]
The map $d$ can be also described by 
\[d_v(\Gamma,o)=\sum_{a+b=N}\sum_{0\leq s< t<r}d_{v;h_s,h_t}^{a,b}(\Gamma,o),\]
where $o=[-;-;-;[h_1,\dots,h_r]\sigma,-]$ and $v=\{h_1,\dots,h_r\}$. Remark the relation
\[d_{v;h^1,h^2}^{a,b}(\Gamma,o)=d_{v;h^2,h^1}^{b,a}(\Gamma,o)\]
for half-edges $h^1\neq h^2\in v$. Here well-definedness of $d$ is proved by the relation with the commutativity relation:
\begin{prop}
Using the notations above, $d_vS_{v,h_r;i}(\Gamma,o)$ is equal to zero under the commutativity relation.
\end{prop}
\begin{proof}
For integers $p,q$, we define the linear ordered set $[p,q]$ by $\{p<p+1<\cdots <q-1<q\}$. If $p>q$, put $[p,q]=\emptyset$. For partial ordered sets $P_1,P_2$, we denote their direct sum by $P_1+P_2$ (in the category of posets), and their ordinal sum by $P_1\oplus P_2$. Then a $(p,q)$-shuffle is equivalent to the inverse of an order-preserving bijection $[1,p]+[p+1,p+q]\to [1,p+q]$. 

Let $\tau^{-1}:[1,i]+[i+1,r-1]\to [1,r-1]$ be an $(i,r-i-1)$-shuffle and $0\leq s<t<r$ integers. Put $L=\tau([s+1,t])$ and $l=t-s$.

If $\tau(s+1),\dots,\tau(t)$ are $\leq i$, then we have $\tau(s+m)=\tau(s+1)+(m-1)$ for $1\leq m\leq t-s$ since $[1,i]\to \tau^{-1}([1,i])$ is an isomorphism between posets. Put $a=\tau(s+1)-1$. Then we obtain the shuffle $\tau_2$ by $\tau$:
\[\xymatrix{[1,i-l+1]+[i-l+2,r-l]\ar[r]^-{\tau_2^{-1}}&[1,r-l]\\
[1,a]\oplus\{*\}\oplus [a+l,i]+[i+1,r-1]\ar[r]^-{\text{bij.}}\ar[u]^{\text{canonical isom.}}&[1,s]\oplus\{*\}\oplus[t+1,r-1]\ar[u]_{\text{canonical isom.}}\\[1,i]+[i+1,r-1]\ar@{->>}[u]\ar[r]^-{\tau^{-1}}&[1,r-1]\ar@{->>}[u]}\]
The shuffle $\tau$ can recover from a pair $(a,l,\tau_2)$, where $\{a+1,\dots,a+l\}\subset [1,i]$ and an $(i-l+1,r-i-1)$-shuffle $\tau_2$.

Similarly, if $\tau(s+1),\dots,\tau(t)$ are $\geq i+1$, we can obtain a triple $(a,l,\tau_2)$, where $\{a+1,\dots,a+l\}\subset [i+1,r-1]$ and an $(i-l+1,r-i-1)$-shuffle $\tau_2$.

Otherwise, put $p=\#(L\cap [1,i])$. Then we obtain the shuffle $\tau_1$ by restricting $\tau$:
\[\xymatrix{[1,p]+[p+1,l]\ar[d]_{\text{canonical isom.}}\ar[r]^-{\tau_1^{-1}}&[1,l]\ar[d]^{\text{canonical isom.}}\\
L\ar@{^{(}-_>}[d]\ar[r]^-{\tau^{-1}}&\tau^{-1}(L)\ar@{^{(}-_>}[d]\\
[1,i]+[i+1,r-1]\ar[r]^-{\tau^{-1}}&[1,r-1]}\]
We consider $\bar{L}=([1,i]+[i+1,r-1])\setminus L$ and the order-preserving bijection $\rho^{-1}:\bar{L}\to [1,s]\oplus [t+1,r-1]$ defined by the restriction of $\tau^{-1}$. The shuffle $\tau$ recovers from a pair $(\rho,\tau_1)$, where $\rho^{-1}:\bar{L}\to [1,s]\oplus [t+1,r-1]$ is an order-preserving bijection and $\tau_1$ is a $(p,l-p)$-shuffle.

Thus we have
\begin{align*}
d_vS_{v,h_r;i}([h_1,\dots,h_r]\sigma)&=\sum_{l=1}^{r-1}\left(\sum_{p=1}^{l-1}\sum_\rho\sum_{\tau_1}o_\rho^{\tau_1^{(v',h')}}+\sum_a\sum_{\tau_2} o_{a,l}^{\tau_2^{(v'',h_r)}}\right)\\
&=\sum_{l=1}^{r-1}\left(\sum_{p=1}^{l-1}\sum_\rho S_{v',h';p}(o_\rho)+\sum_aS_{v'',h_r;i-l+1}(o_{a,l})\right),\end{align*}
where $L=\{1,\dots,r-1\}\setminus \bar{L}=\{u_1<\cdots<u_p\text{ as integers}\}$,
\begin{align*}
o_\rho&=\epsilon[[h_{u_1},\dots,h_{u_p},h']\sigma,[h_{\rho(1)},\dots,h_{\rho(s)},h'',h_{\rho(t+1)},\dots,h_{\rho(r-1)},h_r]\sigma],\\
o_{a,l}&=\epsilon'[[h_{a+1},\dots,h_{a+l},h']\sigma,[h_1,\dots,h_a,h'',h_{a+l+1},\dots,h_r]\sigma],\end{align*}
and $\epsilon,\epsilon'$ are appropriate Koszul signs. (In these equations, the subscriptions $\cyc$ are omitted.)
\end{proof}

\subsubsection{Definition of the second differential $L$} For $\Gamma\in \G(N)$, let $i_v(\Gamma)$ be the $N$-graded graph obtained by converting a normal vertex $v$ of degree $-1$ to a special vertex. We define the linear map $i_v:O_\com(W,\Gamma)\to O_\com(W,i_v(\Gamma))$ for $o\in O_\com(W,\Gamma)$ such that \[i_v(\Gamma,[-;-;-;c,-])=(i_v(\Gamma),[-;-;-,c;-])\] for $c\in C(v)$ if $v$ has degree $-1$ and valency $\geq 3$, and $i_v(\Gamma,o)=0$ if $v$ does not. Since the relation
 \[i_{v_1}S_{v_2,h_r;k}(\Gamma,o)=S_{v_2,h_r;k}i_{v_1}(\Gamma,o)\]
for $v_1,v_2\in V_i(\Gamma)$ holds clearly, the map $i_v$ is well-defined. Then the linear map $L:\tC\to\tC$ is defined by
\[L:=id-di,\]
where the linear map $i:\tC\to \tC$ is obtained by
\[i(\Gamma,o):=\sum_{v\in V_n(\Gamma)}i_v(\Gamma,o).\] 
The map $L$ is also described by
\[L(\Gamma,o)=\sum_{v\in V_n(\Gamma)}(i_{v'}+i_{v''})d_v(\Gamma,o)\]
since $i_ud_v=d_vi_u$ for normal vertices $u\neq v$.

Then $d$, $i$, and $L$ have (cohomological) bidegree $(1,0)$, $(-1,1)$ and $(0,1)$ respectively.


\subsubsection{Definition of the underlying bigraded vector space $\C$} 

The space $\C$ is the quotient space of $\tC$ by
\begin{itemize}
\item \textbf{($A_\infty$-relation)} \[R_v(\Gamma,o):=i_{v'}i_{v''}d_v(\Gamma,o)=0\] for $\Gamma\in\G(N)$ and a normal vertex $v$ (of degree $-2$). 
\begin{figure}[h]
\centering
\includegraphics[width=5.5cm]{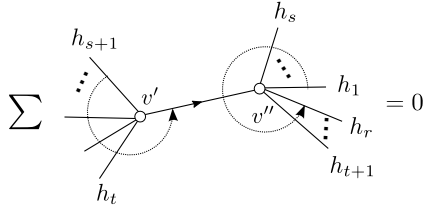}
\caption{$A_\infty$-relation. }
\end{figure}%
\item \textbf{(Cut-off relation)} For $\Gamma\in \G(N)$ and $e=\{h_1,h_2\}\in E(\Gamma)$, we define the $N$-graded graph $\Gamma_e$ as follows:
\[H(\Gamma_e)=H(\Gamma)\amalg \{\bar{h}_1,\bar{h}_2\},\] \[E(\Gamma_e)=(E(\Gamma)\setminus\{e\})\amalg\{\{h_1,\bar{h}_1\},\{h_2,\bar{h}_2\}\},\]
\[V(\Gamma_e)=V(\Gamma)\amalg\{\{\bar{h}_1\},\{\bar{h}_2\}\},\]
\[|\bar{h}_1|=N-|h_1|=:a,\quad |\bar{h}_2|=N-|h_2|=:b.\]
Then
\[(\Gamma,[[h_1,h_2],-;-;-;-])=\sum_{|x^i|=a,|x^j|=b}\omega_{ij}(\Gamma_e,[[h_1,\bar{h}_1],[\bar{h}_2,h_2],-;x^i\sigma^{-1},x^j\sigma^{-1},-;-;-]),\]
where $\{x^i\}$ is a homogeneous basis of $W$ and $(\omega_{ij})$ is the inverse matrix of $(\omega(x^i,x^j))$.
\begin{figure}[h]
\centering
\includegraphics[width=7.5cm]{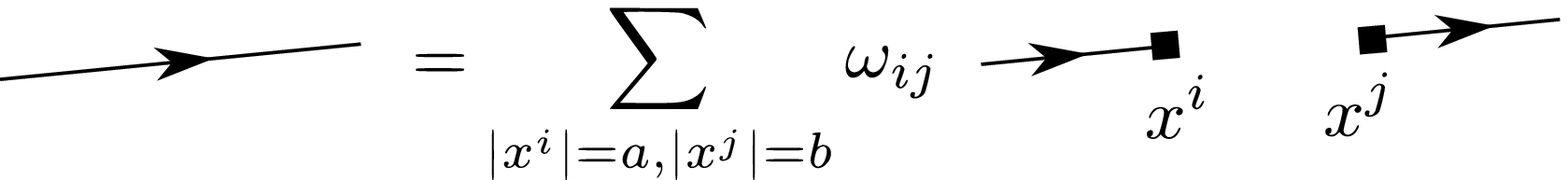}
\caption{Cut-off relation. }
\end{figure}%
\end{itemize}
Remark that $\C$ is generated by $W$-labeled graphs with only one internal vertex by cut-off relation.

\subsubsection{On well-definedness of three operators $d,i,L$ on $\C$} The endomorphisms $d$, $i$ and $L$ of $\tC$ induce endomorphisms of $\C$ by the equations
\[dR_v(\Gamma,o)=\sum_{u\neq v}R_vd_u(\Gamma,o),\quad iR_v(\Gamma,o)=\sum_{u\neq v}R_vi_u(\Gamma,o)\]
for a normal vertex $v$ of an $N$-graded graph $\Gamma$.

\subsubsection{On two differentials $d,L$ on $\C$} 
\begin{prop}\label{dd}
The bigraded vector space $\C$ is a double complex with respect to differentials $d$ and $L$. We call $\C$ \textbf{double graph complex}.
\end{prop}

\begin{proof}
First, we show the equation $d^2=0$. It is proved in the same way as Kontsevich's original graph complex. For a normal vertex $v$ of an $N$-graded graph $(\Gamma,o)$, let $v',v''$ be new vertices obtained by splitting at $v$. Then
\[d_{v'}d_v(\Gamma,o)=-d_{v''}d_v(\Gamma,o)\quad d_ud_v(\Gamma,o)=-d_vd_u(\Gamma,o)\]
for $u\neq v$ holds. The first equation is shown by Figure \ref{split}. In the figure, $v'$ and $v''$ are defined such that the direction of the new edge is from $v'$ to $v''$ in Figure \ref{split}, and $(v')',(v')'',(v'')',(v'')''$ are also defined in the same way. So we obtain  $d^2(\Gamma,o)=0$ by cancellation. 
\begin{figure}[h]
\centering
\includegraphics[width=10cm]{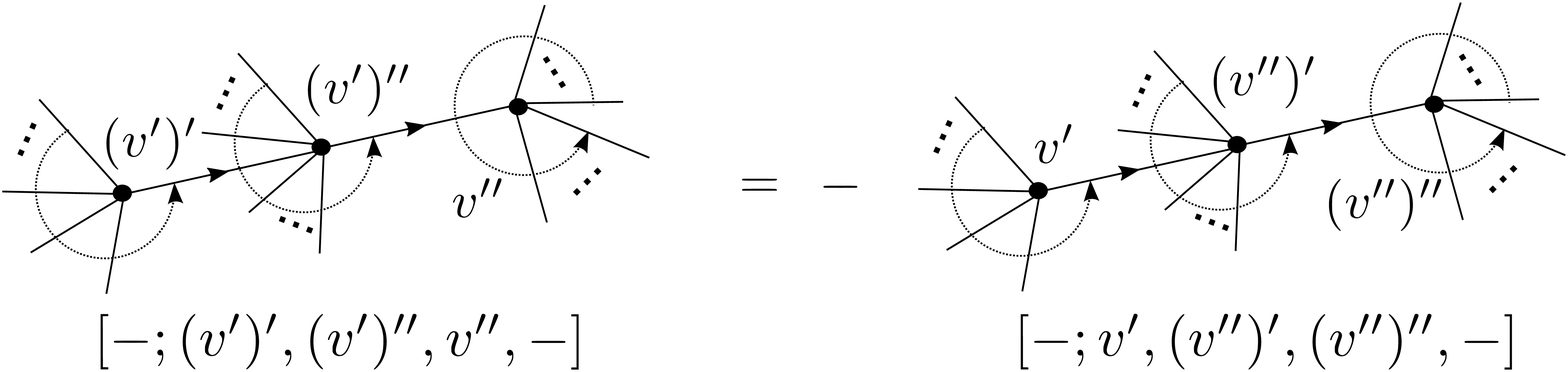}
\caption{$d_{v'}d_v(\Gamma,o)=-d_{v''}d_v(\Gamma,o)$.}
\label{split}
\end{figure}%

Next, we show $L^2=0$. From the equation in $\tC$
\begin{align*}
(iL-Li)(\Gamma,o)
&=\left(\sum_{u}i_u(i_{v'}+i_{v''})d_v-\sum_{u\neq v}(i_{v'}+i_{v''})d_vi_u\right)(\Gamma,o)\\&=\sum_{ v}(i_{v''}i_{v'}+i_{v'}i_{v''})d_v(\Gamma,o)\\
&=2\sum_vR_v(\Gamma,o),\end{align*}
we obtain the relation $iL-Li=0$ in $\C$. So the equations
\[L^2=(id-di)L=idL-diL=idL-dLi=idid-didi,\]
\[L^2=L(id-di)=Lid-Ldi=iLd-Ldi=-idid+didi\]
hold. Then we obtain $L^2=0$. Since $Ld+dL=-did+did=0$ holds by definition of $L$, we get the proposition.\end{proof}

\subsection{Construction of the map to Chevalley-Eilenberg complexes}\label{chainmap}
Let $(W,\omega)$ and $Z$ be as Section \ref{defgraph} and $\delta$ be a symplectic and quadratic differential of homological degree $-1$ on $\ct$. In this section, the Lie algebra $\Der_\omega(\ct)$ of symplectic derivations is denoted by $\D$. We construct a double chain map \[\C\to \CE(\D)\]
from the graph complex $\C$ to the Chevalley-Eilenberg complex of the dgl $(\D,\ad(\delta))$. 


Let $(\Gamma,o)$ be an oriented graph and $\hat{o}$ be a lift of $o$. Put 
\[k=\#V(\Gamma),\quad k_e=\#V_e(\Gamma),\quad k_s=\#V_s(\Gamma),\quad k_n=\#V_n(\Gamma),\]
\begin{align*}
(r_1,\dots,r_k)&:=(\underbrace{1,\dots,1}_{k_e},a_1,\dots,a_{k_s+k_n})\\
&:=(\underbrace{1,\dots,1}_{k_e},\# v^s_1,\dots,\# v_{k_s}^s,\# v_1,\dots,\# v_{k_n})\end{align*}
We denote by $\tau(\hat{o})$ the linear isomorphism (the permutation of factors of the tensor product) 
\[W^{\otimes r_1}\otimes \cdots\otimes W^{\otimes r_k} \to W^{\otimes 2}\otimes \cdots \otimes W^{\otimes 2}=(W^{\otimes2})^{\otimes l}\]
corresponding to the permutation of half-edges
\[(h_1,\dots,h_{k_e},\hat{c}_1^s,\dots,\hat{c}_{k_s}^s,\hat{c}_1,\dots,\hat{c}_{k_n})\mapsto (\hat{o}_1,\dots,\hat{o}_l).\]
Then we define the linear map $\alpha(\Gamma,\hat{o})$ of cohomological degree $(l-k)N$
by composing these maps
\[\alpha(\Gamma,\hat{o}):W[-N]^{\otimes k_e}\otimes \Der_\omega(\ct)^{\otimes (k_s+k_n)} \overset{\text{proj.}}{\to} W[-N]^{\otimes k_e}\otimes \bigotimes_{i=1}^{k_s+k_n}\Der^{a_i+2}_\omega(\ct)  \]
\[\overset{\Phi}{\simeq} W[-N]^{\otimes k_e}\otimes \bigotimes_{i=1}^{k_s+k_n}W(a_i)[-N] \subset\bigotimes_{i=1}^{k}(W^{\otimes r_i}[-N])\overset{\sigma^{\otimes k}}{\to}  \bigotimes_{i=1}^{k}W^{\otimes r_i}\overset{\tau(\hat{o})}{\to }(W^{\otimes2})^{\otimes l}\overset{\omega_E}{\to} \r,\]
where $\Phi:=\id^{\otimes k_e}_{W[-N]}\otimes \Phi_\omega^{\otimes(k_s+k_n)}$, $\omega_E:=\omega_{e_1}\otimes \cdots \otimes \omega_{e_l}$ and $\omega_{e_j}:=\omega_{(|h^{e_j}_1|,|h^{e_j}_2|)}$ if $e_j=\{h^{e_j}_1,h^{e_j}_2\}$. Here we denote by $\omega_{(d_1,d_2)}$ for integers $d_1,d_2$ the composition of the projection $W\otimes W\to W_{d_1}\otimes W_{d_2}$ and the restriction of $\omega$ to $W_{d_1}\otimes W_{d_2}$. The map $\alpha(\Gamma,\hat{o})$ is independent of a choice of linear orders of half-edges representing cyclic orders, and compatible with the commutativity relation. 


We define the map $\hat{\psi}(\Gamma,\hat{o}):\D^{\otimes k_n}\to\r$ by
\[\hat{\psi}(\Gamma,\hat{o})(D_1,\dots,D_{k_n}):=\alpha(\Gamma,\hat{o})(w_1,\dots,w_{k_e},\underbrace{\delta,\dots,\delta}_{k_s},D_1,\dots,D_{k_n})\]
for $D_i\in\D$. Restricting the map\footnote{For a graded vector space $V$, the injective map $\Alt_n:\Lambda^nV\to V^{\otimes n}$ is defined by
\[\Alt_n(v_1\cdots v_n)=\frac{1}{n!}\sum_{\sigma\in\mathfrak{S}_n}\bar{\epsilon}(\sigma)v_{\sigma(1)}\otimes \cdots \otimes v_{\sigma(n)}\]
for $v_1,\dots,v_n\in V$, where $\bar{\epsilon}(\sigma)$ is the corresponding anti-Koszul sign.} on the exterior algebra, we can get the map 
\[\psi(\Gamma,o)=\hat{\psi}(\Gamma,\hat{o})\circ \Alt_{k_n}:\Lambda^{k_n}\D\to\r.\]
The map is independent of a representation $\hat{o}$ of $o$ by the definition of an orientation. So we obtain the map $\psi:\C\to \CE(\D)$. 

\begin{figure}[h]
\begin{center}
\includegraphics[width=9cm]{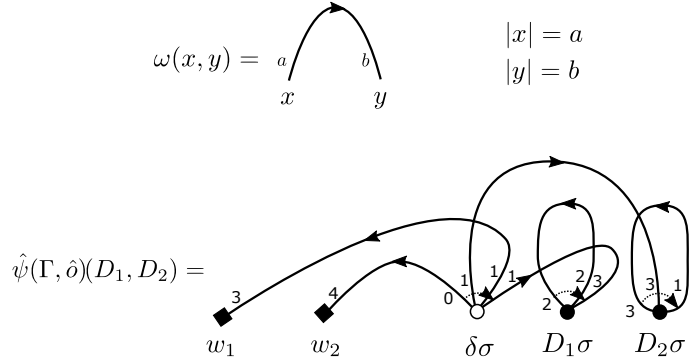}
\caption{An example of $\hat{\psi}(\Gamma,\hat{o})(D_1,D_2)$ ($\Gamma$ is the decorated graph in Figure \ref{ex2}.)}
\label{map-ex}
\end{center}
\end{figure}%

Well-definedness of $\psi$ is proved by the correspondence through $\psi$ between relations in the graph complex $\C$ correspond to properties of derivations as the following table:
\begin{table}[htb]
  \begin{tabular}{c|c}
    graph complex & derivations \\\hline
    cyclicity & symplectic derivation \\
    commutativity & Lie derivation\\
    $A_\infty$-relation & $\delta^2=0$  \\
    cut-off& symplectic form 
  \end{tabular}
\end{table}

By definition, it is clear except for the $A_\infty$-relation. The correspondence for the $A_\infty$-relation is proved in the end of the proof of the following theorem.

\def\St{\mathrm{St}}
\begin{thm}\label{graphmap}
The map $\psi:\C\to \CE(\D)$ is a double chain map.

\end{thm}
\begin{proof}
First, we shall show that $d_{CE}\psi=\psi d$ on $\tC$. To prove this, we need Lemma \ref{derstr}. 

For an oriented graph $(\Gamma,o)$, we define the two lifts $\hat{o}^1$, $\hat{o}^2$ on $\Gamma_{v_i;h_{\nu},h_\mu}^{s,t}$ as follows: 
\[\hat{o}^1=((h',h''),-;-;-;v_1,\dots,v_i',v_i'',\dots,v_p),\]\[ \hat{o}^2=((h'',h'),-;-;-;v_1,\dots,v_i'',v_i',\dots,v_p),\]
\[v_i'=(h_{\nu+1}^{i},\dots,h_{\mu}^{i},h'),\quad v_i''=(h_1^{i},\dots,h_{\nu}^{i},h'',h_{\mu+1}^{i},\dots,h_{r_i}^{i}),\]
where $r_i=\# v_i$. The signs $\epsilon_i$ defined by the equations
\[o^1:=\epsilon_1[\hat{o}^1],\quad o^2:=\epsilon_2[\hat{o}^2],\quad d^
{s,t}_{v_i,h_{\nu},h_\mu}o=(-1)^{i-1}o^1=(-1)^{i-1}o^2.\]
So we obtain
\begin{align*}
d(\Gamma,o)&=\sum_{i=1}^k\sum_{\nu<\mu}\sum_{a+b=N}(-1)^{i-1}(\Gamma_{v_i;h_{\nu},h_\mu}^{a,b},o^1)\\
&=\sum_{i=1}^k\sum_{\nu<\mu}\sum_{a+b=N}(-1)^{i-1}(\Gamma_{v_i;h_{\nu},h_\mu}^{a,b},o^2).
\end{align*}
Note that
\[d_{CE}(\chi\circ \Alt_p)=\frac{1}{2}\sum_{s=1}^p(-1)^{s-1}\chi\circ (1^{\otimes s-1}\otimes [\ ,\ ]\otimes 1^{\otimes p-s})\circ \Alt_{p+1}\]
for a linear map $\chi:W^{\otimes r_1}[-N]\otimes \cdots\otimes W^{\otimes r_p}[-N]\to \r$ and the anti-symmetrization $\Alt_p$ for $p$-components. So 
we should prove
\[ \hat{\psi}(\Gamma,\hat{o})\circ (1^{\otimes i-1}\otimes [\ ,\ ]\otimes 1^{\otimes p-i-1})\]\[=\sum_{\nu<\mu}\sum_{a+b=N}(\epsilon_1\hat{\psi}(\Gamma_{v_i;h_{\nu},h_\mu}^{a,b},\hat{o}^1)+\epsilon_2\hat{\psi}(\Gamma_{v_i;h_{\nu},h_\mu}^{a,b},\hat{o}^2)\circ\tau),\]
where the map $\tau$ means the permutation 
\[X_1\otimes \cdots \otimes (x_{\nu+1}\cdots x_\mu x')\otimes (x_1\cdots x_\nu x''x_{\mu+1}\cdots x_{r_i})\otimes\cdots \otimes X_p\]\[\mapsto \epsilon \cdot X_1\otimes \cdots \otimes (x_1\cdots x_\nu x'x_{\mu+1}\cdots x_{r_i})\otimes (x_{\nu+1}\cdots x_\mu x'')\otimes\cdots\otimes X_p\]
and $\epsilon$ is the Koszul sign. It follows from the equations
\[ \hat{\psi}(\Gamma,\hat{o})\circ (1^{\otimes i-1}\otimes \sigma^{-1}(1\otimes\omega_{(a,b)})\pi_{1;t}^{r',r''}\sigma^{\otimes2}\otimes 1^{\otimes p-i-1})=\epsilon_1\hat{\psi}(\Gamma_{v_i;h_{\nu},h_\mu}^{a,b},\hat{o}^1),\]
\[\hat{\psi}(\Gamma,\hat{o})\circ (1^{\otimes i-1}\otimes \sigma^{-1}(1\otimes\omega_{(a,b)})\pi_{2;t}^{r',r''}\sigma^{\otimes2}\otimes 1^{\otimes p-i-1})=\epsilon_2\hat{\psi}(\Gamma_{v_i;h_{\nu},h_\mu}^{a,b},\hat{o}^2)\circ\tau,\]
for $r'=\mu-\nu+1$, $r''=r-\mu+\nu+1$, and $t=\nu+1$. The first equation is verified as follows: we have by the definition of $\hat{\psi}$
\[\omega(x',x'')\hat{\psi}(\Gamma,\hat{o})(X_1,\dots,X_p)=\epsilon_1\hat{\psi}(\Gamma_{v_i;h_{\nu},h_\mu}^{a,b},\hat{o}^1)(X_1,\dots,X_i',X_i'',\dots,X_p)\]
for $X_s\in W^{\otimes r_s}$, $x'\in W_a$, and $x''\in W_b$. Here we put $X_i'=x_{\nu+1}\cdots x_{\mu}x'\sigma^{-1}$ and $X_i''=x_1\cdots x_{\nu}x''x_{\mu+1}\cdots x_{r}\sigma^{-1}$ for $X_i=x_1\cdots x_r\sigma^{-1}$. So we obtain the first equation from
\[\epsilon_1X_i\omega(x',x'')=\sigma^{-1}(1\otimes\omega)\pi_{1;t}^{r',r''}\otimes\sigma^{\otimes2}(X_i'\otimes X_i'')\]
The second is also verified in the same way. 

Next, we shall prove $i_\delta \psi=\psi i$ on $\tC$. The ordering \[\hat{o}_i:=(-;-;-;v_i,v_1,\dots,\hat{v}_i,\dots,v_p)\] is a lift of $\bar{\epsilon}_i\cdot o$, where $\bar{\epsilon}_i$ is the anti-Koszul sign of the permutation 
\[(v_1,\dots,v_p)\mapsto (v_i,v_1,\dots,\hat{v}_i,\dots,v_p).\] So we have
\begin{align*}
&\psi i(\Gamma,o)(X_{1},\dots,X_{p-1})\\
=&\sum_{s=1}^{j}\bar{\epsilon}_i\cdot\alpha(i_{v_i}(\Gamma),\hat{o}_i)(w_1,\dots,w_{k_e},\underbrace{\delta,\dots,\delta}_{k_s+1},\Alt_{p-1}(X_{1},\dots,X_{p-1}))\\
=&\sum_{s=1}^j\sum_{\pi\in\mathfrak{S}_{p-1}}\bar{\epsilon}\cdot
\alpha(\Gamma,\hat{o})(w_1,\dots,w_{k_e},\underbrace{\delta,\dots,\delta}_{k_s},X_{\pi(1)},\dots,\delta,\dots,X_{\pi(p-1)})\\
=&\alpha(\Gamma,\hat{o})(w_1,\dots,w_{k_e},\underbrace{\delta,\dots,\delta}_{k_s},\Alt_p(\delta,X_{1},\dots,X_{p-1}))\\
=&i_\delta\psi(\Gamma,o)(X_{1},\dots,X_{p-1})
\end{align*}
where $\bar{\epsilon}$ is the anti-Koszul sign of
\[(\delta,X_{1},\dots,X_{p-1})\mapsto (X_{\pi(1)},\dots,\delta,\dots,X_{\pi(p-1)}).\]

From the discussion above, the relation $\psi(R_v(\Gamma,o))=0$ follows from
\[\psi(R_v(\Gamma,o))=\psi(i_{v'}i_{v''}d_v(\Gamma,o))=\psi(\Gamma,o)([\delta,\delta],-)=0.\]
Thus $\psi$ induces the map $\psi:\C\to \CE(\D)$. Furthermore, since $\psi$ is commutative with $d$ and $i$, so is $L$. So we complete the proof.
\end{proof}

The group $\Sp(W,\delta)$ acts on $\C$ by the action on the their labels. Then, the chain map $\psi:\C\to \CE(\D)$ is $\Sp(W,\delta)$-equivariant clearly. Especially we can consider the $\Sp(W,\delta)$-invariant part $\C^{\Sp(W,\delta)}$ of the complex $\C$. It has the double subcomplex $\Ccom(N,Z)$ consisting of $N$-graded graphs which have no external vertex. This complex $\Ccom(N,Z)$ does not depend on the symplectic vector space $W$. It depends only a range $Z$  of degrees and a degree $N$ of a symplectic inner product.

\begin{rem}
We can define the associative version of $\C$ as follows. Set 
\[\tilde{O}_\text{ass}(W,\Gamma):=\bigodot_{e\in E(\Gamma)}O(e)\otimes\bigodot_{u\in V_e(\Gamma)}W[-N]_{|u|}\otimes\bigwedge_{v^s\in V_s(\Gamma)}\Cyc(v^s)[N] \otimes \bigwedge_{v\in V_n(\Gamma)}\Cyc(v)[N], \] 
\[\Cass(W):=\bigoplus_{\Gamma\in\G(N)}O_\text{ass}(W,\Gamma),\quad O_\text{ass}(W,\Gamma):=\tilde{O}_\text{ass}(W,\Gamma)_{\Aut(\Gamma)}.\]
Then $(\Cass(W),d,L)$ is also a double $\Sp(W,\delta)$-chain complex and the chain map 
\[\Cass(W)\to \CE(\Der_\omega(\cct))\]
can be defined in the same way. In this case, we can also consider the double subcomplex $\Cass(N,Z)$ which consists of $N$-graded graphs without external vertices. 
\end{rem}

\section{Applications and examples}

Examples of relations between our chain map and a known notion are written in the following two Examples.

\begin{ex}
For a cyclic minimal $A_\infty$-algebra $(H,I,m)$ with even degree, putting $W:=H^*[-1]$, we have the map $\Cass(W)\to \CE(\Der_\omega(\cct))$. Here $\cct$ is the dual of the bar construction of $(H,I,m)$. The map induced by the chain map
\[C^{0,\bullet}_\text{ass}(N,Z)\to C^{0,\bullet}_{CE}(\Der_\omega(\cct))=\r\]
is known as the Kontsevich cocycle (\cite{K, PS, HL}) of a cyclic $A_\infty$-algebra $(H,I,m)$. 
\end{ex}

\begin{ex}
In the case of $Z=\{0\}$ and $\delta=0$, the chain map 
\[C_\text{ass}^{\bullet,0}(0,\{0\})\to C^{\bullet,0}_{CE}(\Der_\omega(\cct))^{\Sp(W)}\] is equal to Kontsevich's chain map \cite{K, K2}. 
\end{ex}

In the case that $W$ is positively graded, we define a chain complex $\C_+$ by \[\C_+=\C/(\text{positivity}),\] where the positivity relation is as follows:
\begin{itemize}
\item \textbf{(positivity)} \textbf{(i)} a graph which has a normal vertex $v$ satisfying $|v|\leq 0$ is zero, and \textbf{(ii)} $(i_{v'}+i_{v''})d_v(\Gamma,o)=0$ for an oriented graph $(\Gamma,o)$ and a normal vertex $v$ of degree $0$. 
\end{itemize}
The differentials $d,L$ are also defined on $\C_+$, while $i$ is not. 

\begin{prop}
The operators $d,L$ induce the differentials on $\C_+$.

\end{prop}
\begin{proof}
It is clear that these operators are compatible with the former condition \textbf{(i)} of the positivity relation. Note that,  to prove compatibility with $L$ for a graph including a vertex with degree $0$, we need to use \textbf{(ii)}. 

We shall prove they are compatible with \textbf{(ii)}. First, we shall calculate the image of \textbf{(ii)} by the operator $d$. For $\Gamma\in \G(N)$ and a normal vertex $v$ of degree $0$, we have
\begin{align*}
d(i_{v'}+i_{v''})d_v&=d_{v''}i_{v'}d_v+d_{v'}i_{v''}d_v+\sum_{u\neq v',v''}d_u(i_{v'}+i_{v''})d_v\\
&=d_{v''}i_{v'}d_v+d_{v'}i_{v''}d_v-\sum_{u\neq v}(i_{v'}+i_{v''})d_vd_u.
\end{align*}
Here we used the equations in the proof of Theorem \ref{dd}. For a splitting of $v$ such that $|v'|=-1$, $d_{v''}i_{v'}d_v$ must have a non-positive vertex since $|v''|=1$. In the same way, $d_{v'}i_{v''}d_v$ also have a non-positive vertex. So $d(i_{v'}+i_{v''})d_v$ is equal to zero under the positivity relation. 

Next, we shall calculate the image of \textbf{(ii)} by the operator $L$:  
\begin{align*}
L(i_{v'}+i_{v''})d_v=&\sum_{u}(i_{u'}+i_{u''})d_u(i_{v'}+i_{v''})d_v\\
=&(i_{(v'')'}+i_{(v'')''})d_{v''}i_{v'}d_v+(i_{(v')'}+i_{(v')''})d_{v'}i_{v''}d_v\\
&-\sum_{u\neq v}(i_{v'}+i_{v''})d_v(i_{u'}+i_{u''})d_u\\
=&(i_{(v'')'}+i_{(v'')''})i_{v'}d_{v''}d_v+(i_{(v')'}+i_{(v')''})i_{v''}d_{v'}d_v\\
&-\sum_{u\neq v}(i_{v'}+i_{v''})d_v(i_{u'}+i_{u''})d_u.
\end{align*}
By changing names of vertices like the proof of Theorem \ref{dd}, we get
\[
(i_{(v'')'}+i_{(v'')''})i_{v'}d_{v''}d_v=-(i_{(v')''}+i_{v''})i_{(v')'}d_{v'}d_v=-R_{v'}d_{v'}d_v-i_{(v')'}i_{v''}d_{v'}d_v,\]
and 
\begin{align*}
&(i_{(v'')'}+i_{(v'')''})i_{v'}d_{v''}d_v+(i_{(v')'}+i_{(v')''})i_{v''}d_{v'}d_v\\
=&-R_{v'}d_{v'}d_v+i_{(v')''}i_{v''}d_{v'}d_v\\
=&-R_{v'}d_{v'}d_v-i_{(v'')'}i_{(v'')''}d_{v''}d_v\\
=&-R_{v'}d_{v'}d_v-R_{v''}d_{v''}d_v
\end{align*}
Using the $A_\infty$-relation, $L(i_{v'}+i_{v''})d_v$ is equal to zero under the positivity relation. 
\end{proof}

Then we can also get the chain map
\[\psi_+:\C_+\to C^\bullet(\Der^+_\omega(LW))\]
induced by $\psi$.

\begin{ex}
Suppose $X=\#_g(S^{n}\times S^n)\setminus \Int D^{2n}$. Its Quillen model is described by:
\[L_X=L(u_1,v_1,\dots,u_g,v_g)\ (\deg u_i=\deg v_i=n-1),\quad \delta=0,\]
\[\omega(u_i,v_j)=\delta_{ij},\ \omega(u_i,u_j)=\omega(v_i,v_j)=0.\]
It means $N=2n-2$, $W=\braket{u_1,v_1,\dots,u_g,v_g}$ and $Z=\{n-1\}$. Then the dgl $(\Der^+_\omega(L_X),0)$ is a Quillen model of $B\aut_{\pt,0}(X)$ (which is proved in \cite{Ber}). In the case, we can forget all special vertices in the graph complex sicne $\delta=0$. So we have the chain map
\[\Ccom(2n-2,\{n-1\})_+/\text{(special vertices)}\to \CE(\Der^+_\omega(L_X))^{\Sp(W)}.\]
This map is constructed by \cite{Ber} and it is proved that the map is an isomorphism under the limit $g\to \infty$.

\end{ex}

\def\d{\mathfrak{d}}
\begin{ex}
Suppose $X=\c P^3\setminus \Int D^6$. Its Quillen model is described by:
\[L_X=L(u_1,u_2)\ (\deg u_i=2i-1),\quad \delta=\frac12 [u_1,u_1]\frac{\pt}{\pt u_2},\]
\[\omega(u_1,u_2)=\omega(u_2,u_1)=1.\]
It means $N=4$, $W=\braket{u_1,u_2}$ and $Z=\{1,3\}$. Then the dgl $(\Der^+_\omega(L_X),\delta)$ is a Quillen model of $B\aut_{\pt,0}(X)$. Since $\Sp(W,\delta)=1$, we have the chain map
\[\C_+\to \CE(\Der^+_\omega(L_X))=\CE(\Der^+_\omega(L_X))^{\Sp(W,\delta)}.\]
We shall define a certain sub dgl $\d$ of $\Der_\omega(L_X)$. Put 
\[A_1=\frac12[u_1,u_1]\frac{\pt}{\pt u_2},\quad A_2=\frac12[u_2,u_2]\frac{\pt}{\pt u_1}\]
\[B_1=\frac12 [u_1,u_1]\frac{\pt}{\pt u_1}+[u_1,u_2]\frac{\pt}{\pt u_2},\quad B_2=[u_1,u_2]\frac{\pt}{\pt u_1}+\frac12 [u_2,u_2]\frac{\pt}{\pt u_2}.\]
Then we have
\[\delta(A_1)=\delta(B_1)=\delta(B_2)=0,\]
\[ \delta(A_2)=\frac12[[u_1,u_1],u_2]\frac{\pt}{\pt u_1}+\frac12[[u_2,u_2],u_1]\frac{\pt}{\pt u_2}=[A_1,A_2]=-[B_1,B_2]=:C,\]
\[[A_i,B_j]=[A_i,A_i]=[B_j,B_j]=0\ (i,j=1,2),\]
\[\deg A_1=-1,\ \deg A_2=5,\ \deg B_1=1,\ \deg B_2=3,\ \deg C=4.\]
Here we put $\delta(Z):=[\delta,Z]$ for simplicity. By the relation above, 
\[\d:=\braket{A_1,A_2,B_1,B_2,C}=\Der^1_\omega(L_X)\oplus \Der^2_\omega(L_X)\]
is a sub dgl of $\Der_\omega(L_X)$. Its positive truncation $\d^+$ is described by
\[\d^+=\braket{A_2,B_1,B_2,C},\]
\[\delta(A_2)=-[B_1,B_2]=C,\ \delta(B_1)=\delta(B_2)=\delta(C)=0,\]
\[[A_2,B_i]=[A_2,A_2]=[B_i,B_i]=[A_2,C]=[B_i,C]=0\ (i=1,2).\]
Let $x,y_1,y_2,z$ be the suspension of the dual basis of $A_2,B_1,B_2,C$. Then the Chevalley-Eilenberg complex of the dgl $\d^+$ is written by
\[\CE(\d^+)=\Lambda(x,y_1,y_2,z)\ (\deg x=6,\ \deg y_1=2,\ \deg y_2=4,\ \deg z=5),\]
\[dx=dy_1=dy_2=0,\ dz=x-y_1y_2\]
and its total cohomology 
\[H_{CE}^\bullet(\d^+)=\Lambda(x,y_1,y_2)/(x-y_1y_2).\]
Since $\d^+$ is the rank $\leq2$ part of $\Der^+_\omega(L_X)$, the map $H_{CE}^\bullet(\Der^+_\omega(L_X))\to H_{CE}^\bullet(\d^+)$ induced by the inclusion has a section. So non-trivial classes in $H_{CE}^\bullet(\d^+)$ gives non-trivial classes in $H_{CE}^\bullet(\Der^+_\omega(L_X))$. 

The relation $dz=x-y_1y_2$ in the Chevalley-Eilenberg complex is corresponding to the relation in the graph complex $\C_+$ described in Figure \ref{example1}. Here the classes $x$ and $y_1y_2$ corresponds to the first term and the sum of the second and third terms in the figure. Remark that $y_1$ and $y_2$ do not correspond to graphs without external vertices. According to the positivity relation, all the trivalent graphs appearing in the right hand side are cycles since the degrees of two half-edges incident to a permitted bivalent vertex in $\C_+$ must be 3. 

\begin{figure}[h]
\begin{center}
\includegraphics[width=7cm]{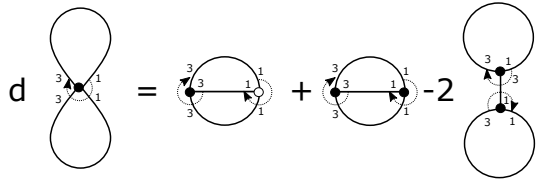}
\caption{the relation of graphs (the orientations are omitted)}
\label{example1}
\end{center}
\end{figure}%

\end{ex}

\begin{ex}
Suppose $X=\c P^4\setminus \Int D^8$. Its Quillen model is described by:
\[L_X=L(u_1,u_2,u_3)\ (\deg u_i=2i-1),\quad \delta=\frac12 [u_1,u_1]\frac{\pt}{\pt u_2}+[u_1,u_2]\frac{\pt}{\pt u_3},\]
\[\omega(u_2,u_2)=\omega(u_1,u_3)=1.\]
It means $N=6$, $W=\braket{u_1,u_2,u_3}$ and $Z=\{1,3,5\}$. Then the dgl $(\Der^+_\omega(L_X),\delta)$ is a Quillen model of $B\aut_{\pt,0}(X)$. Defining the linear transformation $\tau$ by $\tau(u_1)=-u_1$, $\tau(u_2)=u_2$ and $\tau(u_3)=-u_3$, we have $\Sp(W,\delta)=\{1,\tau\}$. So $\C^{\Sp(W,\delta)}$ is generated by graphs labeled by $u_1,u_2,u_3$ satisfying $\#\{u_1,u_3\text{-labeled vertex}\}$ is even. For simplicity, we put 
\[[u_{i_1}\cdots u_{i_k}]:=[u_{i_1},\cdots, u_{i_k}]_\cyc=\sum_{s=1}^k(-1)^{s(k-s)}u_{i_{s+1}}\cdots u_{i_k}u_{i_1}\cdots u_{i_s}\in W_\cyc^k.\]
Using notations in Section \ref{chainmap}, we can take a basis of $W(2)$
\[[u_iu_j]\ (\{i<j\}\subset\{1,2,3\}),\]
a basis of $W(3)$ 
\[\frac13[u_iu_iu_i],\ [u_iu_ju_j],\ [u_iu_iu_j]\ (\{i<j\}\subset\{1,2,3\}),\ [u_1u_2u_3]+[u_1u_3u_2]\]
and a basis of $W(4)$
\[[u_iu_iu_ju_j]\ (i<j),\]\[ [u_1u_1u_2u_3]+[u_1u_1u_3u_2],\ [u_1u_2u_2u_3]-[u_1u_3u_2u_2],\ [u_1u_2u_3u_3]-[u_1u_3u_3u_2].\]
We put the corresponding rank 0, rank 1 and rank 2 basis of $\Der_\omega(L_X)$
\[P_{ij},\ A_{iii},\ A_{ijj},\ A_{iij},\ A_{123},\ B_{iijj},\ B_{1123},\ B_{1223},\ B_{1233},\]
and these dual basis $p_{ij}$, $x_{ijk}$ and $y_{ijkl}$ of $P_{ij}$,\ $A_{ijk}$ and $B_{ijkl}$. Then by direct calculation we have the equations in $\CE(\Der_\omega^+(L_X))$
\begin{align*}
dy_{1122}&=x_{222}-2x_{123}+x_{122}x_{113}-x_{122}x_{122},\\
dy_{2233}&=x_{333}x_{122}+x_{233}x_{222}-x_{223}x_{223}-2x_{123}x_{233}+x_{133}x_{223}+2p_{23}y_{1233},\\
dy_{1133}&=x_{233}-x_{133}x_{113}-x_{123}x_{123}-2p_{23}y_{1123},\\
dy_{1123}&=x_{223}-x_{133}-p_{23}y_{1122},\\
dy_{1223}&=x_{233}+x_{223}x_{122}+x_{123}x_{123}-x_{223}x_{113}-x_{123}x_{222}-x_{133}x_{122}+p_{23}y_{1123},\\
dy_{1233}&=x_{333}+x_{233}x_{122}-x_{123}x_{223}-x_{233}x_{113}-x_{123}x_{122}+p_{23}y_{1133}.
\end{align*}
Here all terms appearing in the right-hand side of the equations are cocycles. For example, the fifth relation is corresponding to the relation in the graph complex $\C_+$ described in Figure \ref{example2}. In Figure \ref{example2}, the image by $\psi_+$ of each graph appearing the last term of the right hand side is zero.

\begin{figure}[h]
\begin{center}
\includegraphics[width=8cm]{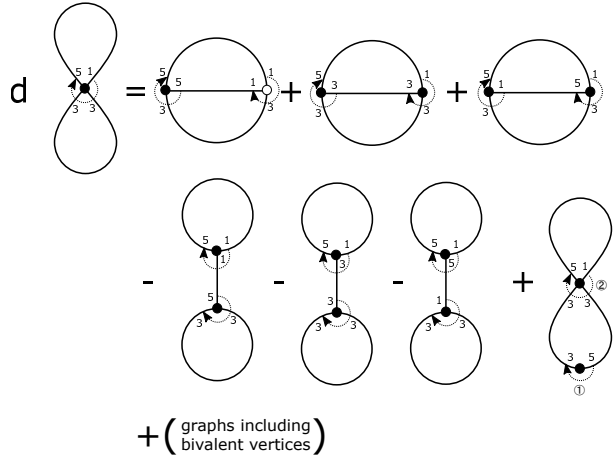}
\caption{the relation of graphs (\textcircled{\scriptsize 1}, \textcircled{\scriptsize 2} mean the orientation of vertices and the other orientations are omitted)}
\label{example2}
\end{center}
\end{figure}%

\end{ex}

\end{document}